\documentclass[a4paper,11pt]{amsart}
\usepackage{amsmath,amsthm,amssymb}
\usepackage{graphicx,color}
\usepackage[all]{xy}
\usepackage{hyperref}
\usepackage[margin=2cm]{geometry}
\usepackage{tikz,enumerate}

\theoremstyle{plain}
\newtheorem{theorem}{Theorem}[section]
\newtheorem{proposition}[theorem]{Proposition}
\newtheorem{lemma}[theorem]{Lemma}

\theoremstyle{definition}

\newtheorem{question}[theorem]{Question}

\DeclareMathOperator{\sign}{sign}

\def\Z{\mathbb{Z}}
\def\Q{\mathbb{Q}}
\def\R{\mathbb{R}}

\def\cF{\cF}
\def\e{\mathbf{e}}

\def\cF{\mathcal{F}}

\def\cL{\mathcal{L}}

\newcommand{\BP}{\mathcal{B}^{odd}}
\newcommand{\CP}{\mathcal{C}^{odd}}
\newcommand{\ZZ}{\Z_2}

\tikzstyle{v}=[circle, draw, solid, fill=black!50, inner sep=0pt, minimum width=4pt]
\tikzstyle{v2}=[circle, draw, solid, fill=black, inner sep=0pt, minimum width=4pt]
\newcommand{\row}{\mathrm{Row}}

\def \K {K}
\def \Kd {K}
\begin{document}
\title[Real toric varieties of Types $C$ and $D$]{The cohomology groups of real toric varieties associated to Weyl chambers of type $C$ and $D$}

\author[S.~Choi]{Suyoung Choi}
\thanks{The first named author was supported by Basic Science Research Program through the National Research Foundation of Korea(NRF) funded by the Ministry of Science, ICT \& Future Planning(NRF-2016R1D1A1A09917654).}
\address[S.~Choi]{Department of mathematics, Ajou University, 206, World cup-ro, Yeongtong-gu, Suwon 16499,  Republic of Korea}
\email{schoi@ajou.ac.kr}

\author[S.~Kaji]{Shizuo KAJI}
\thanks{The second named author was partially supported by KAKENHI, Grant-in-Aid for Young
     Scientists (B) 26800043.}
\address[S.~Kaji]{Department of Mathematical Sciences, Faculty of Science, Yamaguchi University, 1677-1, Yoshida, Yamaguchi 753-8512, Japan / JST PRESTO}
\email{skaji@yamaguchi-u.ac.jp}


\author[H.~Park]{Hanchul Park}
\address[H.~Park]{School of Mathematics, Korea Institute for Advanced Study (KIAS), 85 Hoegiro Dongdaemun-gu, Seoul 02455, Republic of Korea}
\email{hpark@kias.re.kr}

\date{\today}


\maketitle%
\begin{abstract}
Given a root system, the Weyl chambers in the co-weight lattice give rise to a real toric variety, called the real toric variety associated to the Weyl chambers.
We compute the integral cohomology groups of real toric varieties associated to the Weyl chambers of type $C_n$ and $D_n$, completing the computation for all classical types.
\end{abstract}

\maketitle

\tableofcontents

\section{Introduction}

For a root system of rank $n$, its Weyl chambers with the co-weight lattice give rise to an $n$-dimensional non-singular complete fan.
By the fundamental theorem of toric geometry, the fan corresponds to a smooth compact toric variety,
which is called the \emph{toric variety associated to the Weyl chambers}.
The real locus of the toric variety forms a real variety called the \emph{real toric variety associated to the Weyl chambers}.
For each irreducible root system of type $R$, we denote the associated toric variety by $X_R$ and the real toric variety by $X^\R_R$.
In this paper, we are interested in the topology of $X^\R_R$.

In general, since a real toric variety $X^\R$ is the fixed point set of the involution in a toric variety $X$,
the $r$th mod $2$-cohomology group of $X^\R$ is isomorphic to the $2r$th mod~$2$-cohomology group of $X$,
which is completely determined by the number of cones of the corresponding fan.
Therefore, the mod~$2$-Betti numbers of $X^\R_R$ have been known as corollaries of study of $X_R$ such as \cite{Procesi1990,Stembridge1994,Dolgachev-Lunts1994} and \cite{Abe2015}.
However, not so much is known about their rational and integral cohomology groups.
Here, we review the known results in this direction.
For the classical types of $A_n$ and $B_n$, the $r$th $\Q$-Betti numbers of $X^\R_{A_n}$ and $X^\R_{B_n}$ are
computed in \cite{Henderson2012,CPP16}:
\begin{align*}
    \beta^r (X^\R_{A_n};\Q) &= \binom{n+1}{2r}a_{2r}, \quad \text{ and } \\
    \beta^r (X^\R_{B_n};\Q) &= \binom{n}{2r}b_{2r} + \binom{n}{2r-1}b_{2r-1},
\end{align*}
where $\displaystyle \sum_{n=0}^\infty a_n \frac{x^n}{n!} = \sec x + \tan x$ and $\displaystyle \sum_{n=0}^\infty b_n \frac{x^n}{n!} = \frac{1}{\cos x - \sin x}$. Remark that $a_n$ is known as the $n$th Euler zigzag number (see A000111 of \cite{oeis}), and $b_n$ is known as the $n$th generalized Euler number or Springer number (see A001586 of \cite{oeis}). See Table~\ref{table:ab_sequence}.
(We set $a_k = b_k = \binom{n}{k} = 0$ for a negative integer $k<0$ as a convention.)
\begin{table}[h]
  \centering
    \begin{tabular}{|c|ccccccccccc}
      \hline
      $n$ & 0 & 1 & 2 & 3 & 4 & 5 & 6 & 7 & 8 & 9 &$\cdots$ \\ \hline
      $a_n$ & 1 & 1 & 1 & 2 & 5 & 16 & 61 & 272 & 1385 & 7936& $\cdots$ \\
      $b_n$ & 1 & 1 & 3 & 11 & 57 & 361 & 2763 & 24611 & 250737 & 2873041& $\cdots$ \\
      \hline
    \end{tabular}
  \caption{The list of $a_n$ and $b_n$ for small $n$ } \label{table:ab_sequence}
\end{table}

For the exceptional types of $R=G_2$, $F_4$, and $E_6$, the non-zero $\Q$-Betti numbers of $X^\R_R$ have been computed in \cite{Cho-Choi-Kaji2017+}. See Table~\ref{table:exceptional_types}.
\begin{table}[h]
\centering
\begin{tabular}{|c|cccc|c|}
  \hline
  Types $R$ & $\beta^0(X^\R_R)$ & $\beta^1(X^\R_R)$  & $\beta^2(X^\R_R)$  & $\beta^3(X^\R_R)$   \\ \hline
  $G_2$ & 1 & 9 &  &   \\ \hline
  $F_4$ & 1 & 57 & 264 &  \\ \hline
  $E_6$ & 1 & 36 & 1323 & 4392   \\
  \hline
\end{tabular}
  \caption{The list of non-zero Betti numbers of $X^\R_{G_2}$, $X^\R_{F_4}$ and $X^\R_{E_6}$} \label{table:exceptional_types}
\end{table}

It should be noted that the real toric varieties in the above cases have only $2$-torsion in the cohomology groups. Hence, by the universal coefficient theorem, one can compute their integral cohomology groups.

In this paper, we compute the rational Betti number of real toric varieties associated to the Weyl chambers of classical type $C$ and $D$.
Furthermore, we show that they have only $2$-torsion in the cohomology group.
By the universal coefficient theorem, our results completely determine the integral cohomology groups of $X^\R_{C_n}$ and $X^\R_{D_n}$,
thus completing the computation of the integral cohomology groups of the
real toric varieties associated to the Weyl chambers of all classical types.


\begin{theorem}\label{thm:main}
Let $s_m = 2^m -1$ and $t_m = (m-2)2^{m-1}+1$.
The $r$th $\Q$-Betti numbers of $X^\R_{C_n}$ ($n \geq 3$) and $X^\R_{D_n}$ ($n \geq 4$)  are as follows.
\begin{align*}
    \beta^r (X^\R_{C_n}; \Q) =& \binom{n}{2r-2}2^{2r-2}s_{n-2r+2}a_{2r-2} + \binom{n}{2r}(2b_{2r} - 2^{2r}a_{2r}), \quad \text{ and }\\
    \beta^r (X^\R_{D_n}; \Q) =& \binom{n}{2r-4}2^{2r-4}t_{n-2r+4}a_{2r-4} + \binom{n}{2r}(2b_{2r} - 2^{2r}a_{2r}),
\end{align*}
where
$a_k = b_k = \binom{n}{k} = 0$ for a negative integer $k<0$.
Furthermore, $X^\R_{C_n}$ and $X^\R_{D_n}$ have only $2$-torsion in the cohomology groups.
\end{theorem}


Tables~\ref{table:type_c} and \ref{table:type_d} are the lists of non-zero $\Q$-Betti numbers of $X^\R_{C_n}$ and $X^\R_{D_n}$ for $n\leq 11$.
The Euler characteristic of $X$ is denote by $\chi(X)$.

\begin{table}[h]
  \centering
\begin{tabular}{|c|ccccccc|c|}
  \hline
  $n$ & $\beta^0(X^\R_{C_n})$ & $\beta^1(X^\R_{C_n})$ & $\beta^2(X^\R_{C_n})$ & $\beta^3(X^\R_{C_n})$ & $\beta^4(X^\R_{C_n})$ & $\beta^5(X^\R_{C_n})$ & $\beta^6(X^\R_{C_n})$ & $\chi(X^\R_{C_n})$ \\ \hline
  3 & 1 & 13 & 12 &  &  &  &    & 0\\
  4 & 1 & 27 & 106 &  &  &  &   &80\\
  5 & 1 & 51 & 450 & 400 &  &   &  &0 \\
  6 & 1 & 93 & 1410 & 5222 &    &  &  &-3904\\
  7 & 1 & 169 & 3794 & 30954 & 27328 &    &  &0\\
  8 & 1 & 311 & 9436 & 129416 & 474850 &   &  &354560\\
  9 & 1 & 583 & 22572 & 448728 & 3617778 & 3191040   &  &0\\
  10 & 1 &1113& 53040& 1399020& 18908730& 69295142  &  &-51733504\\
  11 & 1 &2157& 123640& 4102164& 80153898& 645241762& 569068544& 0\\

  \hline
\end{tabular}
  \caption{The list of non-zero Betti numbers of $X^\R_{C_n}$ for $n \leq 11$ } \label{table:type_c}
\end{table}

\begin{table}[h]
  \centering
\begin{tabular}{|c|ccccccc|c|}
  \hline
  $n$ & $\beta^0(X^\R_{D_n})$ & $\beta^1(X^\R_{D_n})$ & $\beta^2(X^\R_{D_n})$ & $\beta^3(X^\R_{D_n})$ & $\beta^4(X^\R_{D_n})$ & $\beta^5(X^\R_{D_n})$ & $\beta^6(X^\R_{D_n})$ & $\chi(X^\R_{D_n})$ \\ \hline
  4 & 1 & 12 & 51 & 24 &  &  &   &16\\
  5 & 1 & 20 & 219 & 200 &  &   &  &0 \\
  6 & 1 & 30 & 639 & 2642 & 1200   &  &  &-832\\
  7 & 1 & 42 & 1511 & 15470 & 14000 &    &  &0\\
  8 & 1 & 56 & 3149 & 59864 & 242114 & 109312  &  &76032\\
  9 & 1 & 72 & 6077 & 182472 & 1816146 & 1639680   &  &0\\
  10 & 1 &90& 11237& 479040& 8778330& 35366822& 15955200 &-11101184\\
  11 & 1 &110& 20437& 1143824& 32715210& 324103714& 292512000& 0\\

  \hline
\end{tabular}
  \caption{The list of non-zero Betti numbers of $X^\R_{D_n}$ for $n \leq 11$ } \label{table:type_d}
\end{table}

Only two exceptional cases remain to be open.
\begin{question}
  Compute the integral cohomology groups of $X^\R_{E_7}$ and $X^\R_{E_8}$.
\end{question}

\section{Real toric varieties associated to Weyl chambers}\label{sec:toric-root}
Let $X$ be a smooth compact toric variety and $X^\R$ its real toric variety.
By the fundamental theorem of toric geometry, $X$ is assigned by a non-singular complete fan $\Sigma$.
Assume that $\Sigma$ has $m$ rays, and we put $[m]=\{1,2,\ldots,m\}$ the set of rays of $\Sigma$.
We obtain a (star-shaped) simplicial complex $K$ on $[m]$, whose
simplices correspond to the sets of rays spanning cones of $\Sigma$.
 In addition, the primitive integral vectors of the rays define a map $\lambda$ from $[m]$ to $\Z^n$;
 $\lambda(v)$ is the direction of the ray $v \in [m]$. Note that $\Sigma$ is completely determined by a pair of $K$ and $\lambda$, and, therefore, so is $X$.

Similarly, the real toric variety $X^\R$ is also determined by a pair of $K$ and the composition $\lambda^\R \colon [m] \stackrel{\lambda}{\to} \Z^n \stackrel{\text{mod $2$}}{\longrightarrow} \ZZ^n$, where $\ZZ$ is the field with two elements.
 Throughout the paper, we identify the power set $2^{[m]}$ of $[m]$ with $\ZZ^m$ in the following way: for $S \subset [m]$, the corresponding element $S \in\ZZ^m$ is $\sum_{i\in S} \e_i$, where $\e_i$ is the $i$th standard vector of $\ZZ^m$.
Then, we may regard $\lambda^\R$ as the linear map $\Lambda \colon \ZZ^m \to \ZZ^n$, which is called the \emph{characteristic matrix} of $X^\R$.
We make note of the fact that $X^\R$ does not depend on the choice of a basis of $\ZZ^n$.
Therefore, topological invariants of $X^\R$ should be computed in terms of $K$ and $\Lambda$.
Our method in this paper is based on the following cohomology formulae of real toric varieties.

\begin{theorem}[\cite{ST2012,CP2, Cai-Choi_even_torsion}]\label{thm:cohomology-of-real-toric}
 Let $k$ be a positive integer. Then, as graded modules,
  \begin{align*}
  H^\ast(X^\R;\Q) \cong& \bigoplus_{ S \in \row( \Lambda )} \widetilde{H}^{\ast-1} (K_S; \Q), \\
H^\ast(X^\R;\Z_{2k+1}) \cong& \bigoplus_{ S \in \row( \Lambda )} \widetilde{H}^{\ast-1} (K_S; \Z_{2k+1}), \quad \text{ and } \\
H^\ast(X^\R;\Z_{2^{k+1}}) \cong& \bigoplus_{ S \in \row( \Lambda )} \widetilde{H}^{\ast-1} (K_S; \Z_{2^{k}}),
  \end{align*}
where $\row( \Lambda )$ is the subspace of $\ZZ^m$ spanned by the rows of $\Lambda$, and
$K_S$ is the induced subcomplex of $K$ on the vertices indicated by $S \in 2^{[m]}=\ZZ^m$.
\end{theorem}

In the rest of this section, we review the construction of the non-singular complete fan associated to Weyl chambers (see, for example, \cite{Stembridge1994}).
 Let $V$ be a finite dimensional real Euclidean space, and $\Phi_R \subset V$ an irreducible root system of type $R$. Take a fixed set of simple roots $\Delta_R=\{\alpha_1,\ldots, \alpha_n\}\subset \Phi_R$. Let $\Omega=\{\omega_1,\ldots,\omega_n\}$ be the set of fundamental co-weights, that is, $(\omega_i,\alpha_j)=\delta_{ij}$ with respect to the inner product in the ambient space $V$.
 Then, we can assume $V=\R\langle \Omega \rangle$.
Consider the co-weight lattice $\Z \langle \Omega \rangle = \{ v \in V \mid (v,\alpha) \in \Z \text{ for any } \alpha \in \Phi_R \}$.
Denote by $W_R=\langle s_1,\ldots, s_n \rangle$ the Weyl group generated by the simple reflections acting on $V$. Then, $W_R$ decomposes $V$ into several chambers,
giving a non-singular complete fan in $\R\langle \Omega \rangle$ with the lattice structure $\Z\langle \Omega \rangle \subset \R\langle \Omega \rangle$.
Let $V_R=W_R\cdot \Omega=\{v_1,\ldots,v_m\}$ be the set of rays spanning the chambers, and $K_R\subset 2^{V_R}$ the corresponding simplicial complex called the \emph{Coxeter complex} of type $R$ \cite[\S 1.15]{Humphreys1990}, whose maximal simplices are the chambers.
Then the corresponding characteristic matrix is obtained as $\Lambda_R=(v_1,v_2,\ldots,v_m)$, where the columns are the mod 2 coordinates of the rays
with respect to any basis of $\Z\langle \Omega \rangle$.
Observe that the maximal cones in the fan correspond to sets of simple roots;  for each set of simple roots $\Delta$ in $\Phi_R$ the corresponding cone is $C_\Delta = \{ v \in V \mid (v, \alpha) >0 \text{ for any $\alpha \in \Delta$}\}$. This observation is useful to compute $K_R$ and $\Lambda_R$ explicitly.

We fix some notations which are used later;
      \begin{itemize}
        \item $[n]=\{1, 2, \ldots, n\}$ and $[\pm n] = \{ \pm 1, \pm 2, \ldots, \pm n\}$.
        \item For $I \subset [\pm n]$, $I^+=\{ i\in[n] \mid i\in I \}$.
        \item For $I \subset [\pm n]$, $I^-=\{ i\in[n] \mid -i\in I \}$.
        \item For $I \subset [\pm n]$, $I^{\pm}=I^+\cup I^-$.
      \end{itemize}

We also recall the following facts for the proof of the main theorem.

\begin{lemma}[Section~3 of \cite{Wachs2007}] \label{lem:posets}
The cohomology groups of $\CP_{2r}$ and $\BP_{2r}$ are
\[
\widetilde{H}^*(\CP_{2r};\Z)\cong
\begin{cases} \Z^{b_n} & (*=r-1) \\ 0 & (*\neq r-1),
\end{cases}
\qquad
\widetilde{H}^*(\BP_{2r};\Z)\cong
\begin{cases} \Z^{a_n} & (*=r-1) \\ 0 & (*\neq r-1),
\end{cases}
\]
where
$\CP_{2r}$ is the poset complex of the odd rank-selected lattice of faces of the cross-polytope over $2r$ points,
and $\BP_{2r}$  is the poset complex of the odd rank-selected Boolean algebra on $2r$ points.
\end{lemma}

\section{Type $C_n$}
 In this section, since $\Phi_{C_n} = \Phi_{B_n}$ for $n\leq 2$, we assume that $n \geq 3$.
    The root system $\Phi_{C_n}$ of type $C_n$ consists of $2n^2$ roots
    $$
        \pm 2\varepsilon_i ~(1\leq i\leq n)   \quad \text{ and }  \quad \pm\varepsilon_i \pm \varepsilon_j ~(1\leq i<j \leq n),
    $$ where $\varepsilon_i$ is the $i$th standard vector of $\R^n = V$.
    Then, the co-weight lattice $\Z \langle \Omega \rangle$ is
    $$
        \Z \langle \Omega \rangle = \left\{ \frac{1}{2} \left( \ell_1 \varepsilon_1 + \cdots + \ell_n \varepsilon_n \right) ~\mid ~ \ell_i \in \Z, \text{ and } \ell_i \equiv \ell_j \text{ (mod $2$) for all $i,j$ } \right\}.
    $$
    Choose a basis $\{\epsilon_i\mid 1\le i \le n\}$ of $\Z \langle \Omega \rangle$
    defined by $\epsilon_i := \varepsilon_i$ for $i=1, \ldots, n-1$ and $\epsilon_n:= \frac{1}{2} \left( \varepsilon_1 + \cdots + \varepsilon_n \right)$.

    Any set of simple roots of type $C_n$ is written as
    $$
         \Delta = \{\mu_1 \varepsilon_{\sigma(1)} - \mu_2\varepsilon_{\sigma(2)}, \mu_2 \varepsilon_{\sigma(2)} - \mu_3\varepsilon_{\sigma(3)}, \ldots, \mu_{n-1} \varepsilon_{\sigma(n-1)}- \mu_n\varepsilon_{\sigma(n)}, 2 \mu_n\varepsilon_{\sigma(n)} \},
    $$ where $\mu_j = \pm1$ and $\sigma \colon [n] \to [n]$ is a permutation.
    Note that every line containing a ray in $V_{C_n}$ is the intersection of $n-1$ hyperplanes normal to $\Delta \setminus \{ \alpha \}$ for $\alpha \in \Delta$,
    and the direction of the ray is determined by $\alpha$.
    For $\alpha \in \Delta$, there exists a unique primitive integral vector $\beta_{\alpha,\Delta} \in \Z\langle \Omega \rangle$ such that $(\beta_{\alpha,\Delta}, \alpha') = 0$ for all $\alpha' \in \Delta \setminus \{\alpha\}$ and $(\beta_{\alpha,\Delta}, \alpha)>0$.
    More precisely, $\displaystyle \beta_{\alpha,\Delta} = \sum_{k=1}^{i} \mu_k \varepsilon_{\sigma(k)}$ for $\alpha = \alpha^\sigma_i$,
    where $\alpha^{\mu,\sigma}_i := \mu_i \varepsilon_{\sigma(i)} - \mu_{i+1}\varepsilon_{\sigma(i+1)}$ for $i=1, \ldots, n-1$ and
    $\alpha^{\mu,\sigma_n}:= 2 \mu_n\varepsilon_{\sigma(n)}$.
    We label the ray corresponding to $\beta_{\alpha,\Delta}$ by the subset $I = \{ \mu_1\sigma(1), \ldots, \mu_i \sigma(i) \}$ of $[\pm n]$.
    We note that the label subset $I$ satisfies
    \begin{equation}\label{eqn:I+capI-}
        I^+ \cap I^- = \emptyset.
    \end{equation}
    Conversely, for each $I \subset [\pm n]$ satisfying \eqref{eqn:I+capI-}, one can find $\alpha$ and $\Delta$ such that $\alpha \in \Delta$ and $\displaystyle \beta_{\alpha,\Delta} = \beta_I :=\sum_{ k \in I } \sign(k) \varepsilon_{|k|}$.
    Therefore, we have an identification of $V_{C_n}$ as
    $$
        V_{C_n}:=\{ I \subset [\pm n] \mid I^+ \cap I^- = \emptyset \}.
    $$
    Under this identification, a maximal cone $C_\Delta$ corresponds to a sequence of nested sets $I_1 \subsetneq \cdots \subsetneq I_{n-1} \subsetneq I_n$.

    Now, let us consider the characteristic map $\lambda_{C_n}$.
    For a vertex $I \in V_{C_n}$, $\lambda_{C_n}(I)$ is the primitive vector in $\Z\langle \Omega \rangle$ in the direction of $\beta_I$.
    If neither $n$ nor $-n$ is contained in $I$,
    $$\beta_I = \sum_{ k \in I^+ } \varepsilon_{k} - \sum_{k \in I^-} \varepsilon_{k}= \sum_{ k \in I^+ } \epsilon_{k} - \sum_{k \in I^-} \epsilon_{k}.$$
Otherwise, 
$$
    \beta_I = \left\{
                \begin{array}{ll}
                  {\displaystyle  \left( \sum_{ k \in I^+ \setminus\{n\} } \epsilon_{k} - \sum_{k \in I^-} \epsilon_{k}\right) + 2\epsilon_{n} - \epsilon_1 - \cdots - \epsilon_{n-1}}
, & \hbox{if $n \in I$;} \\
                  {\displaystyle \left( \sum_{ k \in I^+ } \epsilon_{k} - \sum_{k \in I^- \setminus\{n\} } \epsilon_{k} \right) - 2\epsilon_{n} + \epsilon_1 + \cdots + \epsilon_{n-1}}
, & \hbox{if $-n \in I$.}
                \end{array}
              \right.
$$
Since $\beta_I$ is primitive in $\Z\langle \Omega \rangle$ if and only if $|I| \leq n-1$, we see
    $$
        \lambda_{C_n}(I) = \left\{
                             \begin{array}{ll}
                  \beta_I, & \hbox{if $|I| \leq n-1$;} \\
                  \frac{1}{2} \beta_I, & \hbox{if $|I|=n$.}
                             \end{array}
                           \right.
$$
It follows
    \begin{equation}\label{eq:char}
        \lambda^\R_{C_n}(I) = \left\{
                             \begin{array}{ll}
                  {\displaystyle \sum_{k \in I^\pm} \e_k}, & \hbox{if $|I| \leq n-1$ and $n \not\in I^\pm$;} \\
                  {\displaystyle \sum_{k \not\in I^\pm} \e_k}, & \hbox{if $|I| \leq n-1$ and $n \in I^\pm$;} \\
                  {\displaystyle \e_n + \sum_{k \in I^-} \e_k} , & \hbox{if $|I|=n$ and $n \in I$;}\\
                  {\displaystyle \e_n + \sum_{k \in I^+} \e_k} , & \hbox{if $|I|=n$ and $-n \in I$,}
                             \end{array}
                           \right.
    \end{equation}
     where $\e_i$ is the $i$th standard vector of $\ZZ^n$. Let $\Lambda_{C_n}$ be the characteristic matrix of $X^\R_{C_n}$ determined by $\lambda^\R_{C_n}$, that is, $\Lambda_{C_n}$ is an $n \times (3^n-1)$ matrix with elements in $\ZZ$ whose columns are indexed by $I \in V_{C_n}$. Its column indexed by $I$ is denoted by $\lambda^\R_{C_n}(I)$.

    In order to apply Theorem~\ref{thm:cohomology-of-real-toric} to compute the cohomology of $X^\R_{C_n}$,
    we consider $(K_{C_n})_S$ for all $S \in \row(\Lambda_{C_n})$.
    For this, it is convenient to
    identify $\row(\Lambda_{C_n})$ with the power set of $[n]$ in the following way;
    we associate to $\{i_1,\ldots,i_{r}\}\subset [n]$
	the sum of $i_1,\ldots,i_{r}$th rows of $\Lambda_{C_n}$.
        We note the rows of $\Lambda_{C_n}$ are symmetric except for  the $n$th row,
    hence we deal with the case for $n\in S$ and $n \not\in S$ separately.

    For each $S \subset [n]$, the vertex set $V(S)$ of $(K_{C_n})_S$ is separated into two parts
    $V_1(S) = \{ I\in V(S) \mid |I|<n\}$ and $V_2(S) = \{ I\in V(S) \mid |I|=n\}$.
    More precisely,
    $$V_1(S) = \left\{ I\in V_{C_n} \mid
  | I^\pm \cap (S \cup \{n\})| \text{ is odd} \text{ and } |I|<n \right\},$$
    and $V_2(S)$ is as follows.
    \begin{itemize}
      \item When $|S|$ is odd and $n \not\in S$,
$
    V_2(S) = \left\{ I\in V_{C_n} \mid |I^- \cap (S\cup \{n\})|\text{ is odd }\text{ and } |I|=n \right\}.
$
      \item When $|S|$ is even and $n \not\in S$,
$
    V_2(S) = \left\{ I\in V_{C_n} \mid |I^- \cap S|\text{ is odd }\text{ and } |I|=n \right\}.
$
      \item When $|S|$ is odd and $n \in S$,
$
    V_2(S) = \left\{ I\in V_{C_n} \mid |I^- \cap S|\text{ is even }\text{ and } |I|=n \right\}.
$
      \item When $|S|$ is even and $n \in S$,
$
    V_2(S) = \left\{ I\in V_{C_n} \mid |I^- \cap (S \setminus \{n\})|\text{ is even }\text{ and } |I|=n \right\}.
$
    \end{itemize}
    Moreover, $(K_{C_n})_S$ can be regarded as the poset complex of the poset on $V(S)=V_1(S) \cup V_2(S)$
    ordered by inclusion.

\begin{proposition} \label{prop:trivial_case_type_c}
    We have
    $$
        (K_{C_n})_{\emptyset} \simeq \emptyset \quad \text{ and } \quad (K_{C_n})_{\{n\}} \simeq \bigvee^{2^{n}-1} S^0.
    $$
\end{proposition}
\begin{proof}
    The former statement is obvious from the definition.
    Note that, by \eqref{eq:char}, $\lambda^\R_{C_n}(I)$ has the $\e_n$-term if and only if $I$ has $n$ elements.
    Therefore, $(K_{C_n})_{\{n\}}$ consists of $2^n$ distinct points, which proves the latter statement.
\end{proof}

\begin{lemma} \label{lemma:(C_n)S_cong_(C_n)Sn}
    Let $n \ge 3$. For a nonempty subset $S\subset [n-1]$, $(K_{C_n})_S$ is isomorphic to $(K_{C_n})_{S\cup \{n\}}$.
\end{lemma}

\begin{proof}
    Put $T= S \cup \{n\}$.
    We will construct an explicit isomorphism between $(K_{C_n})_S$ and $(K_{C_n})_{T}$.
    For each $I\in V_{C_n}$ and each $x\in [n]$, let $I^{-x}$ be the set obtained from $I$ by reversing the sign of the elements in $I\cap \{x,-x\}$, that is,
    $$
        I^{-x}=\{-i \mid i\in  I\cap \{x,-x\}  \} \cup \{ i\mid i\in I\setminus I\cap \{x,-x\} \}.
    $$
    Note that if $ I\cap \{x,-x\} =\emptyset$ then $I^{-x}=I$.
    For any $I_1, I_2\subset V_{C_n}$ and $x\in [n]$, it is trivial to see that
    $I_1\subset I_2$ if and only if $I_1^{-x} \subset I_2^{-x}$.
    We define a simplicial involution $\varphi$ of $K_{C_n}$ by
    $$
        \varphi(I) := \left\{
                        \begin{array}{ll}
                          I^{-n}, & \hbox{if $|S|$ is odd;} \\
                          I^{-1}, & \hbox{if $|S|$ is even.}
                        \end{array}
                      \right.
    $$
    We show it restricts to a bijection from $V(S)$ to $V(T)$, which proves the assertion.
    For this, it is sufficient to show that $\varphi(I) \in V(T)$ for each $I \in V(S)$.
    We observe that $|I| = |\varphi(I)|$ and $\varphi(I)^{\pm} = I^\pm$.
    Therefore, if $I \in V_1(S)$, then $\varphi(I) \in V_1(T)$.
    Suppose that  $|I|=n$ and $n\in I$.
    Then $|I^{-}\cap S|$ is odd.
    If $|S|$ is odd, then $-n\in \varphi(I)$ and $|I^{+}\cap S|$ is even, and therefore, so is $|\varphi(I)^{+}\cap S|$ since $\varphi(I)^+ =I^{+}$.
    If $|S|$ is even, then $n\in \varphi(I)$ and  $|I^{-}\cap S|$ is odd, and therefore, $|\varphi(I)^{-}\cap S|$ differs from $|I^{-}\cap S|$ by 1, so it is even.
    The case where $|I|=n$ and $-n\in I$ is similar.
    Thus,
    for $I \in V_2(S)$, $\varphi(I)$ is in $V_2(T)$.
\end{proof}

\begin{lemma}\label{lemma:(C_n)1_cong_(C_n)12}
Let $n\geq 3$.
For any subset $S\subset [n-1]$ of odd cardinality and for any $a\in [n-1]\setminus S$,
$(K_{C_n})_S$ is isomorphic to $(K_{C_n})_{S\cup \{a\}}$.
\end{lemma}

\begin{proof}
    Put $T = S \cup \{a\}$, and $|S|=2r-1$ with $r \geq 1$.
    For $I \in V_{C_n}$, we let $\psi(I)$ be the subset of $[\pm n]$ obtained from $I$ by switching $\pm a$ with $\pm n$.
    It is easy to see that $\psi$ defines an involution of $K_{C_n}$.
    Furthermore, one can check that $\psi$ maps $V(S)$ to $V(T)$,
    inducing an isomorphism between $(K_{C_n})_S$ and $(K_{C_n})_{T}$.
\end{proof}

By Lemma~\ref{lemma:(C_n)S_cong_(C_n)Sn}, we may assume $n\not\in S$.
 Also by Lemma~\ref{lemma:(C_n)1_cong_(C_n)12},
 we see $(K_{C_n})_S$ depends only on $\lfloor \frac{|S|+1}{2} \rfloor$ when $n\not\in S$.
From now on, we assume that $S=\{1,2,\ldots,2r-1\}$ with $r\ge 1$ and $2r-1 < n$.
We invoke the famous ``Quillen's Theorem A'' to find a simpler complex which is homotopy equivalent to $(K_{C_n})_S$.

In what follows, we do not distinguish a poset from its order complex by abuse of notation.
Let $K^r_{C_n}$ be a poset on $U \cup W$ ordered by inclusion, where $U = \{ I \in V_1(S) \mid I^\pm \subset S \cup \{n\} \}$ and $W = V_2(S)$.
Note that $K^r_{C_n}$ is a subposet of $(K_{C_n})_{S}$.

\begin{theorem}[Proposition~1.6 of \cite{Quillen1978}] \label{thm:quillen_proposition}
    Let $X$ be a subposet of a poset $Y$.
    If $X_{\le y}$ is contractible for all $y \in Y$, then the inclusion $X\subset Y$ is a homotopy equivalence.
\end{theorem}

\begin{lemma}\label{lem:removal}
We have a homotopy equivalence $(K_{C_n})_{S}\simeq K^r_{C_n}$.
\end{lemma}
\begin{proof}
By Theorem~\ref{thm:quillen_proposition},
we only have to check that $(K^r_{C_n})_{\leq I}$ is contractible for each vertex $I$ of $(K_{C_n})_S$.
If $|I|=n$, then $(K^r_{C_n})_{\leq I}$ is a cone whose apex is $I$.
For $|I|<n$, let $J$ be the maximum subset of $I$ satisfying $J^\pm \subset (S\cup \{n\})$.
Then, $J$ is the unique maximum of $(K^r_{C_n})_{\leq I}$, and $(K^r_{C_n})_{\leq I}$ is a cone whose apex is $I$.
\end{proof}

Now we investigate the cohomology of $K^r_{C_n}$.
For simplicity, we just put $\K = K^r_{C_n}$.

Observe that $\K_U\cong \CP_{2r}$.
In particular, by Lemma~\ref{lem:posets},
\begin{equation} \label{eqn:K_U}
\widetilde{H}^\ast(\K_U)\cong \widetilde{H}^\ast(\CP_{2r})\cong \left\{
                                                            \begin{array}{ll}
                                                              \Z^{b_{2r}} & \hbox{($\ast=r-1$)} \\
                                                              0 & \hbox{($\ast\neq r-1$).}
                                                            \end{array}
                                                          \right.
\end{equation}
Since $\K$ is obtained by attaching $r$-dimensional simplices to $\K_U$,
$\widetilde{H}^\ast(\K) \cong \widetilde{H}^\ast(\K_U)\cong 0$ for $\ast<r-1$ and for $\ast>r$.
We compute $\widetilde{H}^r(\K)$ and $\widetilde{H}^{r-1}(\K)$ in the following.
Let $\cL=\{J=(\pm 1,\pm 2, \ldots, \pm (2r-1), \pm n)\mid |J^-|\text{ is odd }\}$.
Then, $|\cL|=2^{2r-1}$.
For $J\in \cL$,
we see $\K_{<J}\simeq \BP_{2r}$ and $\K_{>J}$ is $2^{n-2r}$ points.
Therefore, for $V_J= \left\{I\in V\mid I\subset J \text{ or } I\supset J \right\}$,
we have
\[
\K_{V_J} \simeq \K_{<J} {\star} \K_{>J} \simeq \BP_{2r} {\star} \bigvee^{2^{n-2r}-1} S^0,
\]
where $X \star Y$ denotes the join of two topological space $X$ and $Y$,
and by Lemma~\ref{lem:posets},
\[
\widetilde{H}^\ast(\K_{V_J};\Z) \cong \begin{cases} \Z^{(2^{n-2r}-1)a_{2r}}  & (\ast=r) \\ 0 & (\ast\neq r).
\end{cases}
\]
Observe that $\K=\bigcup_{J\in \cL}\K_{V_J}$.
Since $|J\cap J'|\le 2r-2$ for any two distinct elements $J$ and $J'$ of $\cL$,
the dimension of $\K_{V_J}\cap \bigcup_{J'\in \cL\setminus \{J\}} \K_{V_{J'}}$ is less than $r-1$.
By the Mayer-Vietoris sequence, we have
\[
\widetilde{H}^r(\K)=\widetilde{H}^r\left(\bigcup_{J\in \cL} \K_{V_{J}}\right) \cong
\bigoplus_{J\in \cL}\widetilde{H}^r(\K_{V_J})\cong
\Z^{2^{2r-1} (2^{n-2r}-1) a_{2r}}.
\]

We now show inductively on the cardinality of $\cL' \subset \cL$ that
\[
\widetilde{H}^{r-1}\left( \bigcup_{J\in \cL' } \K_{V_{J}} \cup \K_U \right) \cong \Z^{b_{2r}-|\cL'|a_{2r}}.
\]
The case when $|\cL'|=0$ follows from \eqref{eqn:K_U}.
The cases when $|\cL'|>0$ are proved by the Mayer-Vietoris sequence
\begin{align*}
0\leftarrow &
\widetilde{H}^{r}\left(  \bigcup_{J\in \cL'} \K_{V_{J}} \cup \K_{U} \right) \oplus \widetilde{H}^{r}\left(  \K_{V_{J'}} \right)
\xleftarrow{\cong}
\widetilde{H}^{r}\left( \bigcup_{J\in \cL'} \K_{V_{J}} \cup \K_{U} \cup \K_{V_{J'}}\right) \\
\xleftarrow{0} & \widetilde{H}^{r-1}\left( (\bigcup_{J\in \cL'} \K_{V_{J}} \cup \K_{U}) \cap  \K_{V_{J'}} \right)
\leftarrow\widetilde{H}^{r-1}\left(  \bigcup_{J\in \cL'} \K_{V_{J}} \cup \K_{U} \right) \oplus \widetilde{H}^{r-1}\left(  \K_{V_{J'}} \right) \\
\leftarrow &
\widetilde{H}^{r-1}\left( \bigcup_{J\in \cL'} \K_{V_{J}} \cup \K_{U} \cup \K_{V_{J'}}\right) \leftarrow
\widetilde{H}^{r-2}\left( (\bigcup_{J\in \cL'} \K_{V_{J}} \cup \K_{U} )\cap  \K_{V_{J'}} \right)=0,
\end{align*}
where  $J'\not\in \cL'$ since
$\widetilde{H}^\ast\left( \bigcup_{J\in \cL'} \K_{V_{J}} \cup \K_{U} \cap  \K_{V_{J'}} \right)
=\widetilde{H}^\ast(\K_{< J'})\cong \begin{cases} \Z^{a_{2r}} & (\ast=r-1) \\ 0 & (\ast\neq r-1) \end{cases}$
and $\widetilde{H}^{r-1}\left(  \K_{V_{J'}} \right)\cong 0$.
The first isomorphism follows from the fact that $\dim(\K_{U})=r-1$ and the previous computation.

We conclude that, by Lemma~\ref{lem:removal},
\[
\widetilde{H}^\ast((K_{C_n})_S)\cong \widetilde{H}^\ast(\K)\cong \begin{cases}
  \Z^{ (2^{n-2r}-1) 2^{2r-1} a_{2r}}  & (\ast=r) \\
 \Z^{b_{2r} - 2^{2r-1} a_{2r}} & (\ast=r-1) \\
               0 & \text{(otherwise).}
               \end{cases}
\]

Combining this with Proposition~\ref{prop:trivial_case_type_c} and Theorem~\ref{thm:cohomology-of-real-toric},
we obtain the type $C_n$ part of Theorem~\ref{thm:main}.

\section{Type $D_n$}
    The (real) toric variety associated to the Weyl chamber of type $D_n$ can be obtained similarly to the Type $C_n$ case. Since $\Phi_{D_n} = \Phi_{A_n}$ for $n\leq 3$, we consider the case when $n \geq 4$.
    The root system $\Phi_{D_n}$ of type $D_n$ consists of $2n(n-1)$ roots
    $$
        \quad \pm\varepsilon_i \pm \varepsilon_j ~(1\leq i<j \leq n),
    $$ where $\varepsilon_i$ is the $i$th standard vector of $\R^n = V$.
    Then, the co-weight lattice $\Z \langle \Omega \rangle$ is
    $$
        \Lambda = \left\{ \frac{1}{2} \left( \ell_1 \varepsilon_1 + \cdots + \ell_n \varepsilon_n \right) ~\mid ~ \ell_i \in \Z, \text{ and } \ell_i \equiv \ell_j \text{ (mod $2$) for all $i,j$ } \right\}.
    $$
    Choose a basis $\{\epsilon_i\mid 1\le i \le n\}$ of $\Z \langle \Omega \rangle$
    defined by $\epsilon_i := \varepsilon_i$ for $i=1, \ldots, n-1$ and $\epsilon_n:= \frac{1}{2} \left( \varepsilon_1 + \cdots + \varepsilon_n \right)$.
    Any set of simple roots of type $D_n$ is written as
    $$
         \Delta = \{\mu_1 \varepsilon_{\sigma(1)} - \mu_2\varepsilon_{\sigma(2)}, \mu_2 \varepsilon_{\sigma(2)} - \mu_3\varepsilon_{\sigma(3)}, \ldots, \mu_{n-1} \varepsilon_{\sigma(n-1)}- \mu_n\varepsilon_{\sigma(n)}, \mu_{n-1} \varepsilon_{\sigma(n-1)} + \mu_n\varepsilon_{\sigma(n)} \},
    $$ where $\mu_j = \pm1$ and $\sigma \colon [n] \to [n]$ is a permutation.
    Put $\alpha^{\mu,\sigma}_i := \mu_i \varepsilon_{\sigma(i)} - \mu_{i+1}\varepsilon_{\sigma(i+1)}$ for $i=1, \ldots, n-1$ and
    $\alpha^{\mu,\sigma}_n:=\mu_{n-1} \varepsilon_{\sigma(n-1)} + \mu_n\varepsilon_{\sigma(n)}$.
    For $\alpha \in \Delta$, there exists a unique primitive integral vector $\beta_{\alpha,\Delta}$ such that $(\beta_{\alpha,\Delta}, \alpha') = 0$ for all $\alpha' \in \Delta \setminus \{\alpha\}$ and $(\beta_{\alpha,\Delta}, \alpha)>0$.
     We label the ray corresponding to $\beta_{\alpha,\Delta}$ by the subset $I$ of $[\pm n]$ satisfying \eqref{eqn:I+capI-} as we did for the case of type $C_n$. Then, we have

\begin{center}
\begin{tabular}{|c|c|c|}
  \hline
  $\alpha$ & $\beta_{\alpha,\Delta}$ & $I$ \\ \hline
  $\alpha^{\mu,\sigma}_i$ {\tiny for $1 \leq i \leq n-2$} & $\displaystyle \sum_{k=1}^{i} \mu_k \varepsilon_{\sigma(k)}$ & $\{ \mu_1\sigma(1), \ldots, \mu_i \sigma(i) \}$ \\ \hline
  $\alpha^{\mu,\sigma}_{n-1}$ & $\displaystyle  \left(\sum_{k=1}^{n-1} \mu_k \varepsilon_{\sigma(k)}\right) -\mu_n \varepsilon_{\sigma(n)}$  &
$\{ \mu_1\sigma(1), \ldots, \mu_{n-1} \sigma(n-1), -\mu_n \sigma(n) \}$  \\ \hline
  $\alpha^{\mu,\sigma}_{n}$ & $\displaystyle \sum_{k=1}^{n} \mu_k \varepsilon_{\sigma(k)}$  &  $\{ \mu_1\sigma(1), \ldots, \mu_n \sigma(n) \}$ \\
  \hline
\end{tabular}
\end{center}
    Conversely, for each $I \subset [\pm n]$ satisfying both \eqref{eqn:I+capI-} and $|I|\neq n-1$, one can find $\alpha$ and $\Delta$
     such that $\alpha \in \Delta$ and $\beta_{\alpha,\Delta} = \beta_I$.
    Therefore, we have an identification of $V_{D_n}$ as
    $$
        V_{D_n} = \{ I \subset [\pm n] \mid I^+ \cap I^- = \emptyset \text{ and } |I| \neq n-1 \}.
    $$
    Under this identification, a maximal cone $C_\Delta$ corresponds to $n$ such subsets $I_1 \subsetneq \cdots \subsetneq I_{n-2} \subsetneq I_{n-1} \cap I_n$.
    This gives the complete combinatorial structure of $K_{D_n}$.

    We note that $V_{D_n} \subset V_{C_n}$. By similar computation to the case of Type~$C_n$, one can see that the characteristic map $\lambda_{D_n}$ is given exactly same as \eqref{eq:char}, that is, $\lambda_{D_n}(I) = \lambda_{C_n}(I)$ for $I \in V_{D_n} \subset V_{C_n}$.
    Hence, the characteristic matrix $\Lambda_{D_n}$ is an $n \times (3^n - 1 - n\cdot 2^{n-1})$  matrix with elements in $\ZZ$ obtained from $\Lambda_{C_n}$ by restricting columns to those corresponding to $V_{D_n}$.

From now on, let us consider the topology of $X^\R_{D_n}$ for $n \geq 4$. We identify again $\row (\Lambda_{D_n})$ with the power set of $[n]$ as we did in the case of Type~$C_n$. In order to compute the cohomology of $X^\R_{D_n}$ it is enough to consider $(K_{D_n})_S$ for all $S \subset [n]$ by Theorem~\ref{thm:cohomology-of-real-toric}.
We note the rows of $\Lambda_{D_n}$ are symmetric except for  the $n$th row,
    hence we deal with the case for $n\in S$ and $n \not\in S$ separately.

For $m\geq0$, put
$$
    t_m = (m-2)2^{m-1}+1.
$$

\begin{proposition}\label{prop:trivial_case_type_d}
 For $n \geq 4$,
$$(K_{D_n})_{\emptyset} \simeq \emptyset \quad \text{ and } \quad (K_{D_n})_{\{n\}} \simeq \bigvee^{t_n} S^1.$$
\end{proposition}
\begin{proof}
  Since the former statement is obvious from the definition, let us prove the latter one. Note that $(K^\R_{D_n})_{\{n\}} = K_W$, where $W = \{I \in V(K^\R_{D^n}) \mid |I|=n \}$. Recall that $I_1$ and $I_2$ in $W$ are connected if and only if $|I_1 \cap I_2| = n-1$. Hence, each vertex has $n$ edges, and there is no $2$-face in $K_W$. Therefore, $K_W$ is homotopy equivalent to $\bigvee^{t_n} S^1$ as desired.
\end{proof}

By the same argument as in Lemmas~\ref{lemma:(C_n)S_cong_(C_n)Sn} and \ref{lemma:(C_n)1_cong_(C_n)12},
we have the following two lemmas.
\begin{lemma}\label{lemma:(D_n)S_cong_(D_n)Sn}
    For a positive integer $n\ge 4$ and for a nonempty subset $S\subset [n-1]$, $(K_{D_n})_S$ is homotopy equivalent to $(K_{D_n})_{S\cup \{n\}}$.
\end{lemma}

\begin{lemma}\label{lemma:(D_n)1_cong_(D_n)12}
For a positive integer $n\ge 4$ and for odd subset $S\subset [n-1]$ for any $a\in [n-1]\setminus S$,
$(K_{D_n})_S$ is homotopy equivalent to $(K_{D_n})_{S\cup \{a\}}$.
\end{lemma}

By Lemmas~\ref{lemma:(D_n)S_cong_(D_n)Sn} and \ref{lemma:(D_n)1_cong_(D_n)12},
we may assume that
$S=\{1,2,\ldots,2r-1\}$ with $r\ge 1$ and $2r-1 < n$.

We note that each maximal simplex of $(K_{D_n})_{S}$ contains exactly one pair of elements $I_1$ and $I_2$
such that $|I_1\cap I_2|=n-1$.
Let $V(S)$ be the vertex set of $(K_{D_n})_{S}$, and
$V' = \{ J \in [\pm n] \mid J = I_1 \cap I_2 \text{ for some $I_1, I_2 \in V(S)$, and }|J|=n-1 \}$.
Similarly to Lemma~\ref{lem:removal}, we have the following lemma.

\begin{lemma}\label{lem:removal_D}
We have a homotopy equivalence $(K_{D_n})_{S}\simeq K^r_{D_n}$, where
$K^r_{D_n}$ is the poset on $U \cup W \cup V'$ ordered by inclusion, where
$$
U= \{ I \in V(S) \mid I^\pm \subset S \cup \{n\} \text{ and } |I|<n-1\}, \quad \text{ and } \quad
W= \{ I \in V(S) \mid
 |I|= n \}.
$$
\end{lemma}
\begin{proof}
We adjoin new vertices labelled with $I_1 \cap I_2$ on the edge between $I_1$ and $I_2$,
then take the subdivision of the simplex.
Then, $(K_{D_n})_{S}$ is homeomorphic to the poset complex of the vertex set $V(S) \cup V'$, where the poset structure is given by the inclusion of vertex labels.
    Then, this poset complex is indeed homotopy equivalent to $K^r_{D_n}$ by     Theorem~\ref{thm:quillen_proposition}.
\end{proof}

Now we investigate the cohomology of $K^r_{D_n}$. We put $\Kd = K^r_{D_n}$ for simplicity.
We first consider the case when $n=2r$ and $2r+1$, where $t_{n-2r}=0$.
In this case, $\Kd$ is exactly same as in the type $C_n$ case.
Thus, we have
\[
\widetilde{H}^\ast(\Kd)\cong \begin{cases}
 \Z^{b_{2r} - 2^{2r-1} a_{2r}} & (\ast=r-1) \\
               0 & \text{(otherwise).}
               \end{cases}
\]
Now we proceed to the case when $n>2r+1$.
One can see $K_U\simeq \CP_{2r}$.
Since $\Kd$ is obtained by attaching $r+1$-dimensional simplices to $\K_U$,
 by Lemma~\ref{lem:posets} we have $\widetilde{H}^\ast(\K) \cong \widetilde{H}^\ast(\K_U)\cong 0$ for $\ast<r-1$ and for $\ast>r+1$.
Let $\cL=\{J=(\pm 1,\pm 2, \ldots, \pm (2r-1), \pm n)\mid |J^-|\text{ is odd }\}$.
Then, $|\cL|=2^{2r-1}$.
For $J\in \cL$,
we see $\K_{<J}\simeq \BP_{2r}$.
Therefore, for $V_J= \left\{I\in V\mid I\subset J \text{ or } I\supset J \right\}$,
we have $\Kd_{V_J}\simeq \Kd_{< J}\star \Kd_{>J}\simeq \BP_{2r} \star \Kd_{>J}$ for $J\in \cL$,
and $\Kd=\bigcup_{J\in \cL}\Kd_{V_J}$.
Furthermore, $\Kd_{>J}$ has $(n-2r)2^{n-1-2r}$ vertices in $V'$
and $2^{n-2r}$ vertices in $W$.
Since each vertex in $V'$ is adjacent to exactly two edges,
the first Betti number of $\Kd_{>J}$ is $2(n-2r)2^{n-1-2r}-((n-2r)2^{n-1-2r}+2^{n-2r})+1=t_{n-2r}$.
In summary, we have $\Kd_{>J}\simeq \bigvee^{t_{n-2r}} S^1$ and
\[
\widetilde{H}^\ast(\Kd_{V_J})\cong\begin{cases} \Z^{t_{n-2r}a_{2r}}  & (\ast=r+1) \\
 0 & (\text{otherwise}).
\end{cases}
\]

We now show inductively on the cardinality of $\cL' \subset \cL$ that
\[
\widetilde{H}^{\ast}\left( \bigcup_{J\in \cL'} \Kd_{V_{J}} \cup \Kd_{U} \right) \cong
\begin{cases}
\Z^{|\cL'|t_{n-2r}a_{2r}} & (\ast=r+1) \\
\Z^{b_{2r}-|\cL'|a_{2r}} & (\ast=r-1) \\
0 & (\text{otherwise}).
 \end{cases}
\]
This follows from the Mayer-Vietoris sequence
{\scriptsize\begin{align*}
0 &=\widetilde{H}^{r}\left( (\bigcup_{J\in \cL'} \Kd_{V_{J}} \cup \Kd_{U}) \cap  \Kd_{V_{J'}} \right)
&\leftarrow
\widetilde{H}^{r+1}\left(  \bigcup_{J\in \cL'} \Kd_{V_{J}} \cup \Kd_{U} \right) \oplus \widetilde{H}^{r+1}\left(  \Kd_{V_{J'}} \right)
&\leftarrow
\widetilde{H}^{r+1}\left( \bigcup_{J\in \cL'} \Kd_{V_{J}} \cup \Kd_{U} \cup \Kd_{V_{J'}}\right) \\
&\leftarrow \widetilde{H}^{r}\left(( \bigcup_{J\in \cL'} \Kd_{V_{J}} \cup \Kd_{U}) \cap  \Kd_{V_{J'}} \right)
&\leftarrow
\widetilde{H}^{r}\left(  \bigcup_{J\in \cL'} \Kd_{V_{J}} \cup \Kd_{U} \right) \oplus \widetilde{H}^{r}\left(  \Kd_{V_{J'}} \right)
&\leftarrow
\widetilde{H}^{r}\left( \bigcup_{J\in \cL'} \Kd_{V_{J}} \cup \Kd_{U} \cup \Kd_{V_{J'}}\right) \\
&\leftarrow
\widetilde{H}^{r-1}\left(( \bigcup_{J\in \cL'} \Kd_{V_{J}} \cup \Kd_{U}) \cap  \Kd_{V_{J'}} \right)
&\leftarrow
\widetilde{H}^{r-1}\left(  \bigcup_{J\in \cL'} \Kd_{V_{J}} \cup \Kd_{U} \right) \oplus \widetilde{H}^{r-1}\left(  \Kd_{V_{J'}} \right)
&\leftarrow
\widetilde{H}^{r-1}\left( \bigcup_{J\in \cL'} \Kd_{V_{J}} \cup \Kd_{U} \cup \Kd_{V_{J'}}\right) \\
&\leftarrow \widetilde{H}^{r-2}\left( ( \bigcup_{J\in \cL'} \Kd_{V_{J}} \cup \Kd_{U}) \cap  \Kd_{V_{J'}} \right)=0.
\end{align*}}
where  $J'\not\in \cL'$ since
$\widetilde{H}^{\ast}\left( ( \bigcup_{J\in \cL'} \Kd_{V_{J}} \cup \Kd_{U} ) \cap  \Kd_{V_{J'}} \right)
=\widetilde{H}^{\ast}(\Kd_{< J'})\cong \begin{cases} \Z^{a_{2r}} & (\ast=r-1) \\ 0 & (\ast\neq r-1) \end{cases}$
and the terms in the second row are all trivial.

By Lemma~\ref{lem:removal_D}, we conclude that
\[
\widetilde{H}^\ast((K_{D_n})_S)\cong \widetilde{H}^\ast(\Kd)\cong
\begin{cases}
\Z^{2^{2r-1}t_{n-2r}a_{2r}} & (\ast=r+1) \\
\Z^{b_{2r}-2^{2r-1}a_{2r}} & (\ast=r-1) \\
0 & (\text{otherwise}).
 \end{cases}
\]

Combining this with Proposition~\ref{prop:trivial_case_type_d} and Theorem~\ref{thm:cohomology-of-real-toric}, we obtain the type $D_n$ part of Theorem~\ref{thm:main}.

\section*{Acknowledgements}
The authors are grateful to Professor Boram Park for her invaluable comments on the initial version of this paper.
\bigskip

\bibliographystyle{amsplain}

\end{document}